\documentclass[11pt]{article}
\usepackage[T1]{fontenc}
\usepackage[utf8]{inputenc}
\usepackage{lmodern}
\usepackage{microtype}
\usepackage[margin=1in]{geometry}
\usepackage{amsmath,amssymb,amsthm,mathtools}
\usepackage{enumitem}
\usepackage{xcolor}
\usepackage[colorlinks=true,linkcolor=blue,citecolor=blue,urlcolor=blue]{hyperref}

\newtheorem{theorem}{Theorem}
\newtheorem{proposition}[theorem]{Proposition}
\newtheorem{corollary}[theorem]{Corollary}
\newtheorem{remark}[theorem]{Remark}
\newtheorem{definition}[theorem]{Definition}

\newcommand{\E}{\mathbb{E}}
\newcommand{\R}{\mathbb{R}}

\title{Utility-Invariant Support Selection and Eventwise Decoupling\
for Simultaneous Independent Multi-Outcome Bets}
\author{Christopher D. Long}
\date{}

\begin{document}
\maketitle

\begin{abstract}
For simultaneous independent events with finitely many outcomes, consider the expected-utility problem with nonnegative wagers and an endogenous cash position. We prove a short support theorem for a broad class of strictly increasing strictly concave utilities. On any fixed support family and at any optimal portfolio with positive cash, summing the active first-order conditions and comparing that sum with cash stationarity yields the exact identity
\[
\frac{\lambda}{K_{\ell}^{(U)}}=\frac{1-P_{\ell,A}}{1-Q_{\ell,A}},
\]
where $P_{\ell,A}$ and $Q_{\ell,A}$ are the active probability and price masses of event $\ell$, $\lambda$ is the budget multiplier, and $K_{\ell}^{(U)}$ is the continuation factor seen by inactive outcomes of that event. Consequently, after sorting each event by the edge ratio $p_{\ell i}/\pi_{\ell i}$, the exact active support is the eventwise union of the single-event supports, and this support is independent of the utility function. The single-event utility-invariant support theorem is already explicit in the free-exposure pari-mutuel setting in Smoczynski and Miles; the point of the present note is that the simultaneous independent-events analogue follows from the same state-price geometry once the right continuation factor is identified.
\end{abstract}

\section{Introduction}
For a single multi-outcome event, the active-set structure of the full Kelly problem is classical; see, among others, Kelly \cite{Kelly1956}, Rosner \cite{Rosner1975}, Smoczynski and Tomkins \cite{SmoczynskiTomkins2010}, and Whelan \cite{Whelan2023}. In state-price language, cash is an implicit all-state position, so active wagers merely top up favorable outcomes beyond the cash floor; see Long \cite{Long2026}. For general increasing concave utility in the free-exposure pari-mutuel setting, Smoczynski and Miles \cite{SmoczynskiMiles2025} prove that the optimal set of horses is independent of utility.

The purpose of this note is to prove the simultaneous independent-events analogue in a short self-contained form. The decisive structural observation is that, for a fixed event, every inactive outcome sees the same background wealth from the rest of the portfolio. This common continuation factor is the multi-event analogue of the single-event cash floor. Once that object is named, the proof is a short stationarity-decomposition argument: sum the active first-order conditions, condition the cash stationarity equation on one event, and subtract.

The result isolates a clean division of labor. Utility curvature affects the \emph{active weights}, but not the \emph{support}. In the simultaneous independent problem, support selection decouples eventwise and is governed by exactly the same threshold $(1-P)/(1-Q)$ as in the single-event problem. This sharpens the structural interpretation of simultaneous Kelly-type optimization beyond the algorithmic literature on many simultaneous bets; see Whitrow \cite{Whitrow2007}.

\section{Model and admissible utilities}
Consider $m\ge 1$ independent events. Event $\ell$ has outcomes $i\in\{1,\dots,n_\ell\}$, subjective probabilities
\[
  p_{\ell i}>0,
  \qquad
  \sum_{i=1}^{n_\ell}p_{\ell i}=1,
\]
and state prices
\[
  \pi_{\ell i}>0.
\]
A portfolio consists of a cash position $c\ge 0$ and nonnegative wager sizes $g_{\ell i}\ge 0$, subject to the bankroll constraint
\begin{equation}\label{eq:budget}
  c+\sum_{\ell=1}^m\sum_{i=1}^{n_\ell}\pi_{\ell i}g_{\ell i}=1.
\end{equation}
If the realized product-state is $X=(X_1,\dots,X_m)$, terminal wealth is
\begin{equation}\label{eq:wealth}
  W(X)=c+\sum_{\ell=1}^m g_{\ell,X_\ell}.
\end{equation}
We restrict attention to portfolios satisfying
\[
  W(x)>0 \qquad \text{for every product-state }x.
\]

\begin{definition}
An \emph{admissible utility} is a function $U:(0,\infty)\to\R$ such that:
\begin{enumerate}[label=\textup{(U\arabic*)},leftmargin=2.4em]
\item $U\in C^1((0,\infty))$,
\item $U'(w)>0$ for every $w>0$,
\item $U$ is strictly concave on $(0,\infty)$.
\end{enumerate}
\end{definition}

The optimization problem is
\begin{equation}\label{eq:mainproblem}
  \max\; \Phi(c,g):=\E\bigl[U(W(X))\bigr]
  \quad\text{subject to \eqref{eq:budget}, } c\ge 0,\ g_{\ell i}\ge 0,\ W(x)>0\ \forall x.
\end{equation}
For each outcome define the edge ratio
\begin{equation}\label{eq:edge}
  r_{\ell i}:=\frac{p_{\ell i}}{\pi_{\ell i}}.
\end{equation}

A support family is a tuple $A=(A_1,\dots,A_m)$ with $A_\ell\subseteq\{1,\dots,n_\ell\}$. We say that $(c,g)$ is \emph{supported on $A$} if
\[
  g_{\ell i}>0 \text{ for } i\in A_\ell,
  \qquad
  g_{\ell i}=0 \text{ for } i\notin A_\ell.
\]
For fixed $\ell$, define the background wealth from all other events by
\begin{equation}\label{eq:Rell}
  R_\ell(X_{-\ell}) := c+\sum_{r\ne \ell} g_{r,X_r}.
\end{equation}
Then on the slice $\{X_\ell=i\}$ one has
\[
  W(X)=g_{\ell i}+R_\ell(X_{-\ell}).
\]
In particular, if $i\notin A_\ell$ then $g_{\ell i}=0$ and the slice wealth is exactly $R_\ell(X_{-\ell})$, independent of which inactive outcome is being tested.

\section{Fixed-support stationarity and reduced costs}
The next proposition records the only fixed-support facts needed for the theorem.

\begin{proposition}[Fixed-support first-order and reduced-cost conditions]\label{prop:kkt}
Fix a support family $A=(A_1,\dots,A_m)$, and let $(c,g)$ be an optimal portfolio for the restriction of \eqref{eq:mainproblem} to portfolios supported on $A$. Assume $c>0$. Then there exists a multiplier $\lambda>0$ such that
\begin{equation}\label{eq:cashstationarity}
  \lambda = \E\bigl[U'(W(X))\bigr],
\end{equation}
and, for every event $\ell$ and every active outcome $i\in A_\ell$,
\begin{equation}\label{eq:activestationarity}
  p_{\ell i}\,\E_{-\ell}\bigl[U'(g_{\ell i}+R_\ell(X_{-\ell}))\bigr]=\lambda\pi_{\ell i}.
\end{equation}
If $j\notin A_\ell$, define
\begin{equation}\label{eq:Kell}
  K_\ell^{(U)}:=\E_{-\ell}\bigl[U'(R_\ell(X_{-\ell}))\bigr].
\end{equation}
Then the one-sided directional derivative in the feasible direction $g_{\ell j}\mapsto g_{\ell j}+\varepsilon$ at $\varepsilon=0^+$ is
\[
  p_{\ell j}K_\ell^{(U)}-\lambda\pi_{\ell j},
\]
and optimality therefore forces the reduced-cost inequality
\begin{equation}\label{eq:reducedcost}
  p_{\ell j}K_\ell^{(U)}\le \lambda\pi_{\ell j}.
\end{equation}
In particular, for fixed $\ell$, every inactive outcome shares the same continuation factor $K_\ell^{(U)}$.
\end{proposition}

\begin{proof}
Restricting to support family $A$ means that the variables $g_{\ell i}$ with $i\notin A_\ell$ are fixed at zero, while $c$ and the active variables $g_{\ell i}$ with $i\in A_\ell$ vary subject to \eqref{eq:budget}. Because $U$ is strictly concave and wealth is affine in the decision variables, the restricted objective is strictly concave. Standard Kuhn-Tucker conditions therefore apply.

Write the Lagrangian
\[
  L(c,g;\lambda)=\Phi(c,g)-\lambda\left(c+\sum_{\ell,i}\pi_{\ell i}g_{\ell i}-1\right).
\]
Because $c>0$, complementary slackness yields stationarity with respect to $c$. Since $\partial W(X)/\partial c=1$ in every state,
\[
  \frac{\partial \Phi}{\partial c}=\E\bigl[U'(W(X))\bigr],
\]
which gives \eqref{eq:cashstationarity}.

If $i\in A_\ell$, only the slice $\{X_\ell=i\}$ contributes to the derivative with respect to $g_{\ell i}$, so
\[
  \frac{\partial \Phi}{\partial g_{\ell i}}
  =p_{\ell i}\,\E_{-\ell}\bigl[U'(g_{\ell i}+R_\ell(X_{-\ell}))\bigr].
\]
Stationarity in the active variable $g_{\ell i}$ yields \eqref{eq:activestationarity}.

If $j\notin A_\ell$, then $g_{\ell j}=0$. The directional derivative of the objective in the feasible direction $g_{\ell j}\mapsto g_{\ell j}+\varepsilon$ at $\varepsilon=0^+$ equals
\[
  p_{\ell j}\,\E_{-\ell}\bigl[U'(R_\ell(X_{-\ell}))\bigr]=p_{\ell j}K_\ell^{(U)}.
\]
The corresponding directional derivative of the Lagrangian is therefore $p_{\ell j}K_\ell^{(U)}-\lambda\pi_{\ell j}$, and optimality implies \eqref{eq:reducedcost}. The common-factor claim follows directly from \eqref{eq:Kell}.
\end{proof}

\section{The threshold identity}
For a support family $A$, define the active probability and price masses of event $\ell$ by
\begin{equation}\label{eq:PQgeneral}
  P_{\ell,A}:=\sum_{i\in A_\ell}p_{\ell i},
  \qquad
  Q_{\ell,A}:=\sum_{i\in A_\ell}\pi_{\ell i}.
\end{equation}
If the outcomes of event $\ell$ are sorted so that
\[
  r_{\ell 1}\ge r_{\ell 2}\ge \cdots \ge r_{\ell n_\ell},
\]
and $A_\ell=\{1,\dots,k_\ell\}$ is a prefix, we abbreviate
\[
  P_{\ell,k_\ell}:=\sum_{i\le k_\ell}p_{\ell i},
  \qquad
  Q_{\ell,k_\ell}:=\sum_{i\le k_\ell}\pi_{\ell i}.
\]

\begin{theorem}[Utility-invariant threshold identity]\label{thm:main}
Fix a support family $A=(A_1,\dots,A_m)$, and let $(c,g)$ be optimal for the restriction of \eqref{eq:mainproblem} to portfolios supported on $A$. Assume $c>0$. Then for every event $\ell$,
\begin{equation}\label{eq:identity}
  \frac{\lambda}{K_\ell^{(U)}}=\frac{1-P_{\ell,A}}{1-Q_{\ell,A}},
\end{equation}
where $\lambda$ is the multiplier from Proposition \ref{prop:kkt} and $K_\ell^{(U)}$ is defined by \eqref{eq:Kell}.
\end{theorem}

\begin{proof}
Fix an event $\ell$. Summing the active stationarity equations \eqref{eq:activestationarity} over $i\in A_\ell$ gives
\begin{equation}\label{eq:activesum}
  \lambda Q_{\ell,A}
  =
  \sum_{i\in A_\ell}p_{\ell i}\,\E_{-\ell}\bigl[U'(g_{\ell i}+R_\ell(X_{-\ell}))\bigr].
\end{equation}
On the other hand, conditioning the cash stationarity equation \eqref{eq:cashstationarity} on $X_\ell$ yields
\begin{align}
  \lambda
  &= \sum_{i\in A_\ell}p_{\ell i}\,\E_{-\ell}\bigl[U'(g_{\ell i}+R_\ell(X_{-\ell}))\bigr]
     +\sum_{j\notin A_\ell}p_{\ell j}\,\E_{-\ell}\bigl[U'(R_\ell(X_{-\ell}))\bigr] \notag\\
  &= \sum_{i\in A_\ell}p_{\ell i}\,\E_{-\ell}\bigl[U'(g_{\ell i}+R_\ell(X_{-\ell}))\bigr]
     +(1-P_{\ell,A})K_\ell^{(U)}.\label{eq:cashdecomp}
\end{align}
Substituting \eqref{eq:activesum} into \eqref{eq:cashdecomp} gives
\[
  \lambda = \lambda Q_{\ell,A} + (1-P_{\ell,A})K_\ell^{(U)},
\]
which rearranges to
\[
  \lambda(1-Q_{\ell,A})=(1-P_{\ell,A})K_\ell^{(U)}.
\]
Dividing by $K_\ell^{(U)}>0$ proves \eqref{eq:identity}.
\end{proof}

\begin{corollary}[Eventwise prefix support and exact decoupling]\label{cor:decoupling}
Assume, in addition, that each event is sorted so that
\[
  r_{\ell 1}\ge r_{\ell 2}\ge \cdots \ge r_{\ell n_\ell}.
\]
Let $(c,g)$ solve the simultaneous problem \eqref{eq:mainproblem} with $c>0$. Then for each event $\ell$ there exists an integer $k_\ell\in\{0,1,\dots,n_\ell\}$ such that
\[
  g_{\ell i}>0 \text{ for } i\le k_\ell,
  \qquad
  g_{\ell i}=0 \text{ for } i>k_\ell.
\]
Moreover,
\begin{equation}\label{eq:stoprule}
  r_{\ell,k_\ell+1}\le \frac{1-P_{\ell,k_\ell}}{1-Q_{\ell,k_\ell}}
  \qquad (k_\ell<n_\ell),
\end{equation}
and the exact support of event $\ell$ is the same support obtained by solving the corresponding single-event problem with the same probabilities and prices. Hence the simultaneous exact support is the eventwise union of the single-event supports.
\end{corollary}

\begin{proof}
Let $A_\ell:=\{i:g_{\ell i}>0\}$. If $i\in A_\ell$, then $g_{\ell i}>0$, so strict concavity implies
\[
  U'(g_{\ell i}+R_\ell(X_{-\ell}))<U'(R_\ell(X_{-\ell}))
\]
for every $X_{-\ell}$. Taking expectations and using \eqref{eq:activestationarity} and \eqref{eq:Kell},
\[
  \frac{\lambda\pi_{\ell i}}{p_{\ell i}}
  =\E_{-\ell}\bigl[U'(g_{\ell i}+R_\ell(X_{-\ell}))\bigr]
  < K_\ell^{(U)}.
\]
Thus every active outcome satisfies
\begin{equation}\label{eq:active-strict}
  r_{\ell i}>\frac{\lambda}{K_\ell^{(U)}}.
\end{equation}
For every inactive outcome $j\notin A_\ell$, the reduced-cost inequality \eqref{eq:reducedcost} gives
\begin{equation}\label{eq:inactive-weak}
  r_{\ell j}\le \frac{\lambda}{K_\ell^{(U)}}.
\end{equation}
Hence, after sorting by decreasing $r_{\ell i}$, the active set of event $\ell$ must be a prefix $\{1,\dots,k_\ell\}$. Applying Theorem \ref{thm:main} to that prefix support yields
\[
  \frac{\lambda}{K_\ell^{(U)}}=\frac{1-P_{\ell,k_\ell}}{1-Q_{\ell,k_\ell}},
\]
so \eqref{eq:inactive-weak} for the next inactive outcome becomes exactly \eqref{eq:stoprule}. This is the same stopping rule as in the single-event problem, so support discovery decouples eventwise.
\end{proof}

\begin{corollary}[Single-event utility-invariant support]\label{cor:singleevent}
For $m=1$, the optimal active support is the greedy prefix determined by
\[
  r_{k+1}\le \frac{1-P_k}{1-Q_k},
\]
independently of the admissible utility $U$.
\end{corollary}

\begin{proof}
When $m=1$, one has $R_1\equiv c$ and therefore $K_1^{(U)}=U'(c)$. Theorem \ref{thm:main} and Corollary \ref{cor:decoupling} then yield the stated cutoff immediately.
\end{proof}

\begin{remark}[One-dimensional reduction in the single-event problem]\label{rem:scalar}
Once the single-event support $A$ is known, the active first-order conditions imply
\[
  U'(W_i)=\frac{\lambda}{r_i} \qquad (i\in A),
\]
so
\[
  W_i=(U')^{-1}\!\left(\frac{\lambda}{r_i}\right) \qquad (i\in A).
\]
\[
  \theta_A:=\frac{1-P_A}{1-Q_A},
\]
Theorem \ref{thm:main} gives $\lambda/U'(c)=\theta_A$, hence
\[
  c=(U')^{-1}\!\left(\lambda\frac{1-Q_A}{1-P_A}\right).
\]
The budget constraint becomes a single scalar equation for $\lambda$:
\[
  (1-Q_A)(U')^{-1}\!\left(\lambda\frac{1-Q_A}{1-P_A}\right)
  +\sum_{i\in A}\pi_i (U')^{-1}\!\left(\frac{\lambda}{r_i}\right)=1.
\]
For CRRA utility $U(w)=w^{1-\gamma}/(1-\gamma)$, this equation is explicit.
\end{remark}

\begin{remark}[Log utility and CRRA as special cases]\label{rem:crra}
For $U(w)=\log w$, Theorem \ref{thm:main} recovers the support threshold in Long \cite{Long2026}. For CRRA utility $U(w)=w^{1-\gamma}/(1-\gamma)$ with $\gamma>0$, the theorem shows that the active support is unchanged across the entire CRRA family: risk aversion changes the active weights, but not the active set.
\end{remark}

\begin{remark}[Boundary case $c=0$]\label{rem:boundary}
When $c=0$, cash stationarity becomes one-sided. More precisely, there is a slack term
\[
  \nu:=\lambda-\E\bigl[U'(W(X))\bigr]\ge 0.
\]
Repeating the proof of Theorem \ref{thm:main} gives
\[
  \lambda(1-Q_{\ell,A})=(1-P_{\ell,A})K_\ell^{(U)}+\nu,
\]
so
\[
  \frac{\lambda}{K_\ell^{(U)}}
  =
  \frac{1-P_{\ell,A}}{1-Q_{\ell,A}} + \frac{\nu}{(1-Q_{\ell,A})K_\ell^{(U)}}.
\]
Thus the exact utility-invariant threshold identity is an interior phenomenon. This clarifies why exogenous exposure constraints or fractional-betting constraints generally destroy support invariance.
\end{remark}

\begin{remark}[Independence as the structural boundary]\label{rem:independence}
The proof uses independence at exactly one point: for an inactive outcome $j\notin A_\ell$, the continuation factor
\[
  K_\ell^{(U)}=\E_{-\ell}[U'(R_\ell(X_{-\ell}))]
\]
must be the same for every such $j$. Under dependence, the conditional law of $X_{-\ell}$ given $X_\ell=j$ varies with $j$, so different inactive outcomes see different continuation factors. The conditioning identity \eqref{eq:cashdecomp} then no longer collapses to a single scalar threshold. Independence is therefore not a mere convenience; it is the exact structural boundary of the theorem.
\end{remark}

\begin{remark}[Strict vigorish and canonical uniqueness]\label{rem:vigorish}
Suppose every event has strict overround,
\[
  \sum_{i=1}^{n_\ell}\pi_{\ell i}>1 \qquad \text{for every }\ell.
\]
Then the fair-event cash-shift degeneracy is absent: one cannot move a constant amount from cash into every outcome of a single event without changing the budget. If, in addition, there are no exact threshold ties, then the eventwise prefix support in Corollary \ref{cor:decoupling} is unique. If equality occurs at a threshold, the natural canonical tie-break is to prefer cash, i.e. to treat equality outcomes as inactive.
\end{remark}

\end{document}